\documentclass[12pt, one side,emlines]{amsart}
\usepackage[a4paper,margin=1in]{geometry}
\usepackage{fancyhdr}
\usepackage{enumerate}
\usepackage{graphicx}
\usepackage{amssymb,latexsym,xy,eucal,mathrsfs}
\textwidth=18cm \textheight=27cm \theoremstyle{plain}
\usepackage{tikz}
\usetikzlibrary{positioning}
\usepackage{amsmath}
\usepackage{mathtools}
\theoremstyle{definition}

\newtheorem{theorem}{Theorem}[section]
\newtheorem{thm}[theorem]{Theorem}
\newtheorem{lem}[theorem]{Lemma}

\newtheorem{corollary}[theorem]{Corollary}

\usepackage{array}

\begin{document}

	\title{On Euler-Sombor Energy of Graphs}\maketitle 

	\markboth{ Sopan Bansode Sharad Barde Ganesh Mundhe}{On Euler-Sombor Energy of Graphs}\begin{center}\begin{large} Sopan Bansode$^1$, Sharad Barde$^2$, Ganesh Mundhe$^3$,\end{large}\\\begin{small}\vskip.1in\emph{
1. Anantrao Thopte College Bhor,Pune-412206 ,MH,India\\
2. STE's Sinhgad College of Science, Pune-411046, MH,India\\
3. Army Institute of Technology, Pune-411015, MH,India}\\ 
		e-mail: \texttt{1.sopanb211@gmail.com,2.sharadbardew@gmail.com, 3.gmundhe@aitpune.edu.in,}\end{small}\end{center}\vskip.2in

	\begin{abstract} In 2024, Gutman et al. \cite{I.Gutman 3} defined a new molecular descriptor called as The Euler-Sombor  $(ES)$ index of graph. By using this index we define the Euler-Sombor $(ES)$ matrix of a graph $G$ whoes $(i,j)^{th}$ entry is $\sqrt{{d_i}^2+{d_j}^2+d_i.d_j}$ if vertex $v_i$ is adjacent to vertex $v_j$, otherwise $0$. The $ES$ eigenvalues of the graph $G$ are the eigenvalues of its $ES$ matrix ,$ES(G)$. In this paper we discus $ES$ eigenvalues $\nu_i$ and energy of some classes of graphs.\end{abstract}\vskip.2in
	\noindent\begin{Small}\textbf{Mathematics Subject Classification}:
		05C50    \\\textbf{Keywords}: Adjacency Matrix $A(G)$, ES matrix, ES eigenvalues,Forgotten index,Second Zagreb Index, topological index \end{Small}\vskip.2in
	
	\vskip 0.15cm
\baselineskip 19truept 
\section{Introduction}
 \indent The graphs considered here all are  undirected,finite and simple. Let $G(V,E)$ be a graph with vertex set $V(G) =\{v_1, v_2, \dots, v_n\}$ and edge set $ E(G) =\{e_{1},...,e_{m}\}$. We denote $d_i$ the degree of vertex $v_i$ in $G$. Let $K_n$, $C_n$ and $S_n$ denote a complete, cycle and star graph resp. on $n$ vertices also the complete bipartite graph on $m+n$ vertices is denoted by $K_{m,n}$. 
 In chemical graph theory and mathematical chemistry, a molecular graph is a visual representation of a chemical compound where the vertices/nodes represent the atoms, and the edges signify the chemical bonds between them. This graph-based approach provides a structural formula of the compound.
 
 A topological index known as a connectivity index, is a numerical value derived from the molecular graph. It captures various aspects of the graph's topology, which helps in characterizing the molecule's structure. The toplogical indices are invariant under graph isomorphisms, meaning they do not change if the graph is transformed in certain ways.
 
 Topological indices \cite{I.Gutman 2} describes the molecular structure quantitatively, offering insights into the compounds property such as reactivity and stability. Examples of such indices include the Wiener index, Zagreb indices, and Randic index. Each of these indices highlights different structural features which are useful to study the physical and chemical behaviours of the molecule.\\
 \indent The vertex based Sombor index \cite{I.Gutman 2} introduced by I.Gutman is 
 \begin{align*}
  SO(G) = \sum_{e=(v_{i},v_{j})\in E(G)}\sqrt{d_{i}^2+d_{j}^2}
  \end{align*}
 Now recently Gutman($2024$) obtained the Euler-Sombor index \cite{I.Gutman 1} closely related to the Sombor index.
 \begin{align*}
  ES=ES(G) = \sum_{e=(v_{i},v_{j})\in E(G)}\sqrt{d_{i}^2+d_{j}^2+d_{i}.d_{j}}
  \end{align*}
 In $2015$ B. Furtula et al. \cite{B.Furtula et.al.} gives Forgotten index  of a graph $G$ as,
 \begin{align*}
  FI=FI(G) = \sum_{e=(v_{i},v_{j})\in E(G)}({d_{i}^2+d_{j}^2})
  \end{align*}
The Zagreb indices \cite{I.Gutman et.al.,K.Das} of a graph $G$  are defined as,
   \begin{align*}
    Second ~Zagreb ~Index=SZ=SZ(G) = \sum_{e=(v_{i},v_{j})\in E(G)}{d_{i}.d_{j}}
    \end{align*}
Let the $ES$ matrix of the graph G, which a square matrix of order $n$ whose entries are given by 
 \begin{equation*}
 ES(G)_{ij} = \begin{cases}\sqrt{d_{i}^2+d_{j}^2+d_{i}.d_{j}}\hspace{0.5cm}   \text{if}~~(v_{i}, v_{j})\in E(G),\\
 0\hspace{1.9cm} \text{otherwise}.
 \end{cases}
 \end{equation*}
Suppose $\nu_{1}(G)$, $\nu_{2}(G)$, $\dots$, $\nu_{n}(G)$ are the eigenvalues of $ES$ matrix of $G$. Then $ES$ energy of graph \cite{Gutman 1, R.Bal} is given by
\begin{align*}
 E_{ES}(G) = \sum^{n}_{i=1}|\nu_{i}(G)|.
\end{align*}
\indent  
 In this paper, we will explore properties of the $ES$ matrix, $ES$ eigenvalues and $ES$ energy of a graph. The second Section contains preliminaries. In Section 3, we establish some properties of $ES$ eigenvalues of a graph. While in section 4 we obtained some results on $ES$ energy of graph G.

\section {Preliminaries}
In this section, we begin with some well known results from \cite{Bap}. The following Lemma \ref{1}, Lemma \ref{2}, Lemma \ref{3}, and Lemma \ref{21} gives the eigenvalues of $A(G)$ related with eigenvalues of adjacency matrix of graphs $K_{n}$, $K_{m,n}$, $C_{n}$, and $S_n$. The $\lambda_{1}, \lambda_{2}, \dots, \lambda_{n}$ are adjacency eigenvalues of the graph $G$.
It is known that, if all entries of $A$ are strictly positive, then we say $A$ is positive matrix we write $(A)_{ij} > 0$.
If $A$ is a real and symmetric matrix, then all eigenvalues of $A$ are real.
We say that as eigenvalue is simple if its algebraic multiplicity is $1$.\\
\indent The $ES$ matrix of graph $G$ is an irreducible non-negative symmetric real matrix having all $ES$ eigenvalues are real with  trace $0$. For the following results see Bapat \cite{Bap} and Brouwer \cite{Brouwer AE}.
\begin{lem}\label{1}\cite{Bap}
For any positive integer $n$, the eigenvalues of complete graph $K_{n}$ are $\lambda_{1} = n-1$ and $\lambda_{2} = \dots = \lambda_{n} = -1$.
\end{lem}
\begin{lem}\label{2}\cite{Bap}
For any positive integer $m,n$, the eigenvalues of complete 
bipartite graph $K_{m,n}$ are $\lambda_{1} = \sqrt{mn}$, $\lambda_{2} =  \lambda_{3} = \dots = \lambda_{m+n-2} = 0$ and $\lambda_{m+n} = -\sqrt{mn}$.
\end{lem}
\begin{lem}\label{3}\cite{Bap}
For $n \geq 2$, the eigenvalues of $G = C_{n}$ are $\lambda_{i} = 2 cos\frac{2\pi i}{n}$, where $i = 1, \dots, n$.
\end{lem}
\begin{lem}\label{4}\cite{Brouwer AE}
For any positive integer $n$, the eigenvalues of star graph $S_{n}$ are $\lambda_{1} = \sqrt{n-1}$, $\lambda_{2} = \dots = \lambda_{n-1}=0$ and $\lambda_{n} = -\sqrt{n-1}$.
\end{lem}

\section {ES-Eigenvalues of Some Classes of Graph} 
In this section, we obtain $ES$-eigenvalues of some classes of graphs like $r$-regular, $K_{m,n}$, and $S_n$ graphs, we also obtain some properties of the $ES$ matrix of a graph. 
In the following result, we give the $ES$ eigenvalues of $r$-regular graph.
\begin{thm}\label{5}
Let $G$ be a simple connected graph of order $n\geq 3$ with eigen values, $\lambda_{1}\geq \lambda_{2}\geq \dots \geq \lambda_{n}$  and $\nu_{1}\geq \nu_{2}\geq \dots \geq \nu_{n}$ be its $ES$ eigenvalues.
If $G$ is $r$-regular graph, then the $ES$ eigenvalues of $G$ are $\nu_{i} = r\sqrt{3}$ $\lambda_{i}$ for $i = 1, 2, \dots, n$.
\end{thm}
\begin{proof}
Let $G$ be a connected $r$-regular graph of order $n$. 
\\We know that $ES(G)$ = $r\sqrt{3}$ $A(G)$.
\\Thus $\nu_{i} = r\sqrt{3}$ $\lambda_{i}$  \hspace*{0.2in}for $i = 1, 2, \dots, n$ 
\end{proof}

\begin{corollary}\label{6}
If $G = K_n$ complete graph, then the $ES$ eigenvalues of $G$ are $\nu_{1}$ = $(n-1)^2\sqrt{3}$ and $\nu_{2} = \dots = \nu_{n} = -(n-1)\sqrt{3}$. 
\end{corollary}
\begin{proof} 
From Lemma \ref{1}, if $G = K_n$, $\lambda_{1} = n-1$ and $\lambda_{2} = \dots = \lambda_{n} = -1$.
\\Since $ES(G)$ = $(n-1)\sqrt{3}$ $A(G)$.
\\Therefore $\nu_{i} = (n-1)\sqrt{3}$ $\lambda_{i}$  \hspace*{0.2in}for $i = 1, 2, \dots, n$
\end{proof} 
\begin{corollary}\label{7}
If $G = C_n$, then the $ES$ eigenvalues of $G$ are $\nu_{i}$ $=$ ($4\sqrt{3})cos(\frac{2\pi i}{n})$, where $i = 0, 
\dots, n-1$. 
\end{corollary}
\begin{proof} 
From Lemma \ref{2}, if $G=C_n$, $\lambda_{i} = 2 cos\frac{2\pi i}{n}$, for $i = 1, \dots, n$, then $\nu_{i} = 2\sqrt{3}\lambda_{i}$, for $i = 1, \dots, n$, would yield required result.
\end{proof}
The following result gives the $ES$ eigenvalues of complete bipartite graph.
\begin{thm}\label{8}
Let $G$ be a complete bipartite graph of order $m+n$,with $\lambda_{1}\geq \lambda_{2}\geq \dots \geq \lambda_{m+n}$ eigenvalues and $\nu_{1}\geq \nu_{2}\geq \dots \geq \nu_{m+n}$ be its $ES$ eigenvalues. Then $ES$ eigenvalues of $K_{m,n}$ are  $\sqrt{m^2+n^2+mn}\hspace*{0.05in}\lambda_{i}$,\hspace*{0.1in}for $i =1,\dots,m+n$.
\end{thm}
\begin{proof}
From  the ${ES} (K_{m,n})$ $=$ ($\sqrt{m^2+n^2+mn}$) $A(G(K_{m,n}))$.\\ By Lemma \ref{3}, the eigenvalues of  $A(K_{m,n})$  are,
$\lambda_{1} = \sqrt{mn}$, $\lambda_{2} =  \lambda_{3} = \dots = \lambda_{m+n-1} = 0$ and $\lambda_{m+n} = -\sqrt{mn}$.\\
 Therefore, $\nu_{1} = (\sqrt{m^2+n^2+mn})\lambda_{1}$, $\nu_{2} = \nu_{3} = \dots = \nu_{m+n-1}=(\sqrt{m^2+n^2+mn})\lambda_{m+n-1}$ and $\nu_{m+n} = (\sqrt{m^2+n^2+mn})\lambda_{m+n}$.\\ Hence the  desired result obtained.
\end{proof}
The following result gives the $ES$ eigenvalues of star graph $S_{n}$.
\begin{thm}\label{9}
For any positive integer $n$, the $ES$ eigenvalues of star graph $S_{n}$ are $\nu_{i} = (\sqrt{n^2-n+1})\lambda_{i}$
\end{thm}
\begin{proof}
The ${ES} (S_{n}) = (\sqrt{n^2-n+1}) A(G(S_{n}))$.\\
 From Lemma \ref{21}, the $A(S_{n})$ eigenvalues are
$\lambda_{1} = \sqrt{n-1}$, $\lambda_{2} = \dots = \lambda_{n-1}=0$ and $\lambda_{n} = -\sqrt{n-1}$ \\ then the $ES$ eigenvalues of $S_{n}$ are , $\nu_{1} = (\sqrt{n^2-n+1})\lambda_{1}$, $\nu_{2} = \dots = \nu_{n-1}=0$ with multiplicity $n-2$ and $\nu_{n} = (\sqrt{n^2-n+1})\lambda_{n}$.
\end{proof}
The following result gives the relation between sum of squares of all eigenvalues of $ES$ matrix of $r$-regular graph in Forgotten and  Second Zagreb indices of the graph.
\begin{thm}\label{10}
Let $G$ be a connected regular graph of order $n\geq 3$ and $\nu_{1} \geq \nu_{2} \geq \dots \geq \nu_{n}$ be the eigenvalues of ${ES}(G)$. Then 
\begin{align*} \label{}
\sum^{n}_{i=1}\nu_{i}^{2} \leq (n-1)({FI+SZ}),
\end{align*}
\end{thm}
\begin{proof}
By the definition of the ${ES}(G)$ matrix, the trace of ${ES}(G)$ is  $0$.\\
 Therefore, we have $\sum^{n}_{i=1}\nu_{i} = trace ({ES}(G)) =0$.
\begin{align*}
\sum^{n }_{i=1}\nu_{i}^{2} = trace ({ES}(G))^{2} = (n-1) \sum_{v_{i}v_{j}\in E(G)}({d_{i}^2}+{d_{j}^2}+{d_{i}d_{j}}) 
\end{align*}
\begin{align*}
\leq (n-1)\left(\sum_{v_{i}v_{j}\in E(G)}({d_{i}^2}+{d_{j}^2})+\sum_{v_{i}v_{j}\in E(G)}({d_{i}d_{j}})\right).
\end{align*}
\begin{align*}
\sum^{n }_{i=1}\nu_{i}^{2} \leq (n-1)\left(\sum_{v_{i}v_{j}\in E(G)}({d_{i}^2}+{d_{j}^2})+\sum_{v_{i}v_{j}\in E(G)}({d_{i}d_{j}})\right).
\end{align*}
\begin{align*}
\sum^{n }_{i=1}\nu_{i}^{2} \leq (n-1)(FI+SZ)
\end{align*}
For  $r$-regular graph $G$ degree of every vertex is $r$, which gives  
\begin{align*}
\sum^{n}_{i=1}\nu_{i}^{2}=trace({ES}(G)^{2})\leq(n-1) \left(\sum_{v_{i}v_{j}\in E(G)}({d_{i}^2}+{d_{j}^2})+\sum_{v_{i}v_{j}\in E(G)}({d_{i}d_{j}})\right) = r(FI+SZ).
\end{align*}
\end{proof}
\begin{thm}\label{11}
Let $G$ be a connected graph of order $n\geq 3$ with $m$ edges and $\nu_{1} \geq \nu_{2} \geq \dots \geq \nu_{n}$ be the eigenvalues of ${ES}(G)$. Then
\begin{align*} 
\sum^{n}_{i=1}\nu_{i}^{2} = 2(FI+SZ)
\end{align*}
and 
\begin{align*} 
\sum^{n}_{1\leq i < j \leq n}\nu_{i} \nu_{j} = -(FI+SZ).
\end{align*}
\end{thm}
\begin{proof}
By the definition of the ${ES}(G)$ matrix, $\sum^{n}_{i=1}\nu_{i} = trace ({ES}(G)) =0$. \\
As $G$ be a connected,
\begin{align*}
\sum^{n}_{i=1}\nu_{i}^{2} = trace ({ES}(G)^{2})
 = 2 \sum_{v_{i}v_{j}\in E(G)}({d_{i}^2}+{d_{j}^2}+{d_{i}d_{j}})
\end{align*}
\begin{align*}
\sum^{n}_{i=1}\nu_{i}^{2} = 2\left(\sum_{v_{i}v_{j}\in E(G)}({d_{i}^2}+{d_{j}^2})+(\sum_{v_{i}v_{j}\in E(G)}{d_{i}d_{j}})\right)
\end{align*}
\begin{align*}
\sum^{n}_{i=1}\nu_{i}^{2} = 2(FI+SZ)
\end{align*}
Hence, 
\begin{align*}
\sum^{n}_{i=1}\nu_{i}^{2} = 2(FI+SZ).
\end{align*}
Moreover,
\begin{align*}
\sum^{n}_{1\leq i < j \leq n}\nu_{i} \nu_{j} = \frac{1}{2} \left(\left(\sum^{n}_{i=1}\nu_{i}\right)^{2} - \sum^{n}_{i=1}\nu_{i}^{2}\right) = -(FI+SZ).
\end{align*}
\end{proof}
\noindent\textbf{Perron-Frobenius Theorem:}\label{Per-Frob}
If all entries of $n*n$ matrix $A$ are positive,then it has unique maximal eigenvalue and corresponding eigenvector has positive entries.

\begin{lem}\label{30} \cite{R.Liu}
Let $A$ be an $n\times n$ irreducible non-negative symmetric real matrix, $n\geq 2$. Let $\alpha_{1}$ be the maximum eigenvalue of $A$ and ${\bf x}$ be the unit Perron-Frobenius eigenvector of $A$. $A$ has $s$ $(2\leq s \leq n)$ distinct eigenvalues if and only if there exits $s-1$ real numbers $\alpha_{2}, \dots, \alpha_{s}$ with $\alpha_{1} > \alpha_{2} > \dots > \alpha_{s}$ such that
\begin{equation*}
\prod^{s}_{i=2}(A-\alpha_{i}I) = \prod^{s}_{i=2}(\alpha_{1}-\alpha_{i}){\bf{xx}}^{t}.
\end{equation*}
Moreover, $\alpha_{1} > \alpha_{2} > \dots > \alpha_{s}$ are exactly the $s$ distinct eigenvalues of $A$.
\end{lem}

 Now as $ES$ matrix, $ES(G)$ of a connected graph G, which is an irreducible non-negative symmetric real matrix. Then by lemma  \ref{30} ,we have
\begin{lem} \label{31}
Let $G$ be a connected graph of order $n\geq 2$ and ${ES}(G)$ be its $ES$ matrix. Let $\nu_{1}$ be the largest eigenvalue of ${ES}(G)$ and $\bf{x}$ be its corresponding unit column eigenvector. Then ${ES}(G)$ has $k$ $(2\leq m \leq n)$ distinct eigenvalues if and only if there exist $m-1$ real numbers $\nu_{2}, \nu_{3}, \dots, \nu_{m}$ with $\nu_{1} > \nu_{2} > \nu_{3} > \dots > \nu_{m}$ such that
\begin{equation}
\prod^{m}_{i=2}({ES}(G)-\nu_{i}I) = \prod^{m}_{i=2}(\nu_{1}-\nu_{i}){\bf{xx}}^{t}.
\end{equation}
Moreover, $\nu_{1}, \nu_{2}, \nu_{3}, \dots, \nu_{m}$ are exactly the $m$ distinct eigenvalues of ${ES}(G)$. 
\end{lem}

\begin{thm} \label{} 
Let $G$ be a connected graph of order $n\geq 3$. Then $G$ has two distinct the $ES$ eigenvalues if and only if $G=K_n$. 
\end{thm}
\begin{proof}
Suppose $G$ has two distinct ES eigenvalues, $\nu_{1} > \nu_{2}$. Then by lemma \ref{31}, \\
\begin{equation}
 \ ES(G) =(\nu_{1}-\nu_{2}){\bf{xx}}^{t}+ \nu_{2}I,\\
\end{equation}
 From above equation the off diagonal elements of $ES(G)$ are all positive.\\
 Gives $G$ is complete graph ,$K_n$ i.e. $G=K_n$ \\
  Conversely,if $G$ is complete graph,then \\
 \begin{equation}
  \ ES(G) =(n-1)\sqrt{3} A(G) 
 \end{equation}
thus
 \begin{equation}
  \nu_{i} =(n-1)\sqrt{3} \lambda_{i} ,i=1,2,.....,n.
 \end{equation}\\
 Hence $G$ has  two distinct $ES$ eigenvalues ,$(n-1)^{2}\sqrt{3}$ and $-(n-1)\sqrt{3}$ 
\end{proof}
\begin{lem}\label{16}\cite{Bap}
Let$G$ be connected bipartite graph. If $\lambda$ is eigenvalue of $G$ with multiplicity $k$ then $ -\lambda $ is also eigenvalue of $G$ with same multiplicity $k$. i.e. eigenvalues are symmetric about origin. 
\end{lem}

\begin{thm}\label{17}
A connected bipartite graph $G$  of order $n \geq 3$ has three distinct $ES$ eigenvalues if and only if $G$ is a complete bipartite graph.
\end{thm}
\begin{proof}
By Perron-Frobenius theorem,$\nu_{1}>0$ ,$\nu_{n}$ be simple and $\nu_{i}\neq \nu_{1},\nu_{n}$ be three distinct $ES$ eigenvalues of $G$.
such that,$\nu_{i}\neq 0$ \\
AS $G$ is connected bipartite graph.Then by lemma \ref{16}\\
Then $G$ will have 4-distinct eigenvalues.\\
Contradiction to that $\nu_{i}\neq 0$ $ES$  eigenvalue of $G$.\\
Therefore , $\nu_{i} = 0$ is $n-2$ -times $ES$ eigenvalue of $G$\\
Gives $G$ has 3-distinct $ES$ eigenvalues viz.  $\nu_{1}>0 ,\nu_{n}$ and $\nu_{i}=0 ,i\neq 1,n$ \\
Thus by theorem \ref{8},G is complete bipartite graph.\\
Conversly, Let $G$ is complete bipartite graph of order $n\geq 3$ then obviously $G$ is bipartite graph with three distinct ES eigenvalues.
\end{proof}

The following result gives the relation between diameter of graph $G$ and $ES$ eigenvalues of the graph $G$ for this we have following lemma.
\begin{lem}\label{18}
Let $G$ be a connected graph with adjacency matrix $A(G)$.Then the number of different paths of length $k$ from vertices $i$ and $j$ is $(i,j)^{th}$ entry of $(A(G))^{k}$
\end{lem}
\begin{thm}\label{19}
Let $G$ be a connected graph of order $n$ with $m$ distinct $ES$ eigenvalues.  Then diam $(G)\leq m-1$.
\end{thm}
\begin{proof}
Let ${ES}(G)$ be the $ES$ matrix of $G$ and $\nu_{1}\geq \nu_{2}\geq \dots \geq\nu_{n}$ be its $m$ distinct eigenvalues.\\
 Let ${\bf x}$ be the unit eigenvector of ${ES}(G)$ corresponding to the largest eigenvalue $\nu_{1}$. Where  ${\bf x}$ is positive vector. Now, from Theorem \ref{12} it follows that,
\begin{equation*}
\prod^{m}_{i=2}({ES}(G)-\nu_{i}I) = ({ES}(G))^{m-1} + b_{1}({ES}(G))^{m-2} + \dots + b_{m-2}({ES}(G)) + b_{m-1} I =  \prod^{m}_{i=2}(\nu_{1}-\nu_{i}){\bf{xx}}^{t} = N.
\end{equation*}
Observe that $(N)_{ij}>0$ for each $i, j = 1, 2, \dots, n.$  Therefore, for $i \neq j$, there is a positive integer $k$ with $1\leq k \leq m-1$ such that $(({ES}(G))^{k})_{ij}>0$,                 which implies that is a path of length $k$ between vertices $v_{i}$ and $v_{j}$, that is, diam$(G)\leq m-1$. This complete the proof.
\end{proof}
 Now we obtain the lower bound of largest $ES$ eigenvalue of the graph as follows. 
\begin{thm} \label{20}
Let $G$ be a connected graph of order $n\geq 3$ and $m$ edges. Then
\begin{equation} \label{}
\nu_1 > \sqrt{\frac{2(FI+SZ)}{n}}.
\end{equation}
\end{thm}
\begin{proof}
By Theorem (\ref{11}), We have $\sum^{n}_{i=1}\nu_{i}^{2} = 2(FI+SZ)$.\\
Also $\sum^{n}_{i=1}\nu_{i}^{2}$ $< n \nu_1^{2}$, this implies that 
\begin{align*}
n \nu_1^{2} >\sum^{n}_{i=1}\nu_{i}^{2} = 2(FI+SZ)
\end{align*}
\begin{align*}
n \nu_1^{2} > 2(FI+SZ)
\end{align*}
\begin{align*}
\nu_1 > \sqrt{\frac{2(FI+SZ)}{n}}.
\end{align*}
\end{proof}
Now, we will obtain upper bound of largest $ES$ eigenvalue of a graph. 
\begin{thm} \label{21}
Let $G$ be a connected graph of order $n\geq 3$. Then
\begin{equation} 
\nu_1 \leq \sqrt{\frac{2(n-1)(FI+SZ)}{n}}
\end{equation}  
 equality holds (\ref{17}) if $G =K_n$.
\end{thm} 
\begin{proof}
Let $G$ be a connected graph  of order $n\geq 3$. Then by Theorem(\ref{23}),
we have,
\begin{align*}
\sum^{n}_{i=1}\nu_{i}^{2} = 2 (FI+SZ)
\end{align*}
Thus,
\begin{align*}
\nu_{1}^{2}+s\sum^{n}_{i=2}\nu_{i}^{2}= 2(FI+SZ)
\end{align*}
\begin{align*}
\ 2(FI+SZ)-\nu_{1}^{2}= \sum^{n}_{i=2}\nu_{i}^{2}
\end{align*}
Now by using Cauchy-Schwarz  inequality
\begin{align*}
\ 2(FI+SZ)-\nu_{1}^{2}= \sum^{n}_{i=2}\nu_{i}^{2}\geq \frac{1}{(n-1)}\left(\sum^{n}_{i=2}|\nu_{i}|\right)^{2} \geq \frac{\nu_{1}^2}{(n-1)}
\end{align*}
Then,
\begin{align*}
\ 2(FI+SZ)-\nu_{1}^{2}\geq \frac{\nu_{1}^2}{(n-1)}
\end{align*}
gives,
\begin{align*}
\ 2(FI+SZ)\geq\nu_{1}^2\left(1+\frac{1}{n-1}\right)
\end{align*} 
\begin{align*}
\ 2(FI+SZ)\geq \nu_{1}^2\left(\frac{n}{n-1}\right)
\end{align*}
\begin{align*}
 \nu_{1}^2\leq\left(\frac{2(n-1)(FI+SZ)}{n}\right)
\end{align*}
\begin{align*}
\nu_{1}\leq\sqrt{\frac{2(n-1)(FI+SZ)}{n}}
\end{align*}
The equality holds if $G=K_{n}$.
\end{proof}
\section {ES-Energy of Graphs}
\indent In this section, we obtain $ES$ energy of some classes of graph and bounds on the $ES$ energy of a graph.
\begin{thm}
Let $G$ be a $r$-regular connected graph of order $n\geq 3$. Then $E_{ES}(G) = {r}{\sqrt{3}} E(G)$.
\end{thm}
\begin{corollary}
If $G=K_n$ of order $n\geq 3$, then $E_{ES}(K_n) = {\sqrt3(n-1)} E(K_n)$. 
\end{corollary}
\begin{corollary}
If $G=C_n$ of order $n\geq 3$, then $E_{ES}(C_n) = (2{\sqrt{3}}) E(C_n)$. 
\end{corollary}
\begin{corollary}
If $G = K_{m,n}$ of order ($m+n$), then $E_{ES}(K_{m,n}) = (\sqrt{m^2+n^2+mn}) E(K_{m,n})$. 
\end{corollary}
The following result gives the lower bound of $ES$ energy of a graph in the form of Forgotten and Second Zagreb index of graph. 
\begin{thm}
Let $G$ be a connected graph with order $n\geq 3$ and $m$ edges. Then $E_{ES}(G) \geq 2 \sqrt(FI+SZ)$. 
\end{thm}
\begin{proof}
By the definition $E_{ES}(G)$ and Theorem (\ref{23}), we have
\begin{align*}
(E_{ES}(G))^{2} = \sum^{n}_{i=1}|\nu_{i}|^{2} + 2\sum^{n}_{1\leq i < j \leq n}|\nu_{i} \nu_{j}| 
\end{align*}
\begin{align*}
\geq \sum^{n}_{i=1} \nu_{i}^{2} + 2|\sum^{n}_{1\leq i < j \leq n} \nu_{i} \nu_{j}| 
\end{align*}
\begin{align*}
= 2(FI+SZ) + 2|-({FI+SZ})| = 4 (FI+SZ)
\end{align*}
Implies, \\
\begin{align*}
E_{ES}(G) \geq 2 \sqrt{(FI+SZ)}
\end{align*}
\end{proof}
Now the following result gives the upper bound of $ES$ energy of a graph in the form of Forgotten and Second Zagreb index of graph.
\begin{thm}
Let $G$ be a connected graph with order $n\geq 3$ with $m$ edges then $E_{Es}(G) \leq \sqrt{2n(FI+SZ)}$. 
\end{thm}
\begin{proof}

By the definition $E_{Es}(G)$ and Theorem (\ref{23}), we have
\begin{align*}
(E_{Es}(G))^{2} = \sum^{n}_{i=1}|\nu_{i}|^{2} + 2\sum^{n}_{1\leq i < j \leq n}|\nu_{i} \nu_{j}| 
\end{align*}
\begin{align*}
(E_{ES}(G))^{2} \leq |4 (FI+SZ)| \leq 2n(FI+SZ)
\end{align*}
then,\\
\begin{align*}
E_{ES}(G) \leq \sqrt{2n(FI+SZ)}.
\end{align*}
\end{proof}
Now to obtain the upper bound for Euler Sombor energy of the path $P_n$ we use the following lemma.
\begin{lem}\label{}\cite{Bilal Ahmad}
Let $M_{1}$ and $M_{2}$ be the square matrices of order $n$. Then 
\begin{align*}
\sum^{n}_{i=1}\nu_{i}(M_{1}+M_{2})\leq \sum^{n}_{i=1} \nu_{i}(M_{1})+ \sum^{n}_{i=1} \nu_{i}(M_{2})
\end{align*}
equality holds if and only if for an orthogonal matrix $M$ ,such that $M.M_{1}$ and $M.M_{2}$ are positive  semi-definite.
\end{lem}
\begin{thm}
Let $P_{n}$ be a path with order $n\geq 4$ ,then it's Euler Sombor energy is given by 
\begin{align*}
\ E_{ES}(P_{n}) \leq \sqrt{12}{E(P_{n})+2(\sqrt{7}-\sqrt{12})}.
\end{align*} 
\end{thm}
\begin{proof}
The Euler Sombor matrix of the path $P_{n}$ can be written as,\\

$\ {ES}(P_{n}) = (\sqrt{12})
 \begin{bmatrix}
0 & 1 & 0 &\dots & 0 & 0 & 0 \\
1 & 0 & 1 &\dots & 0 & 0 & 0 \\
0 & 1 & 0 &\dots & 0 & 0 & 0 \\
\vdots &\vdots &\vdots &\ddots &\vdots &\vdots &\vdots\\
0 & 0 & 0 &\dots & 0 & 1 & 0 \\
0 & 0 & 0 &\dots & 1 & 0 & 1 \\
0 & 0 & 0 &\dots & 0 & 1 & 0 \\
\end{bmatrix}
+ 
\alpha
\begin{bmatrix}
0 & 1 & 0 &\dots & 0 & 0 & 0 \\
1 & 0 & 0 &\dots & 0 & 0 & 0 \\
0 & 0 & 0 &\dots & 0 & 0 & 0 \\
\vdots &\vdots &\vdots &\ddots &\vdots &\vdots &\vdots\\
0 & 0 & 0 &\dots & 0 & 0 & 0 \\
0 & 0 & 0 &\dots & 0 & 0 & 1 \\
0 & 0 & 0 &\dots & 0 & 1 & 0 \\
\end{bmatrix} \\${\tiny}
where,$\alpha = (\sqrt{7}-\sqrt{12})$ \\
Therefore $\ {ES}(P_{n})=\sqrt{12}\ A(P_{n})+ B $ \\
The Euler Sombor eigen values of $B$ are,
$\pm\{\sqrt{7} -\sqrt{12}\}$ repeated twice \\
Then by lemma $4.7$,Euler Sombor Energy of path $P_{n}$ is,
\begin{align*}
\ E_{ES}(P_{n}) \leq \sqrt{12}{E(P_{n})+2(\sqrt{7}-\sqrt{12})}.
\end{align*}
\end{proof}   
 {\bf Here we have Euler Sombor ,Forgotten and Second Zagreb energies of Octane isomers} \\
\begin{table}[h]
\centering
\caption{Euler Sombor ,Forgotten and Second Zagreb energies of octane isomers}
\begin{tabular}{|p{1cm}|p{5cm}|p{2cm}|p{2cm}|p{2cm}|}
\hline
Sr.No. & Molucule & $E_{ES}$ & $E_{FI}$ & $E_{SZ}$ \\
\hline
1 & octane & 30.2172 &66.2439& 31.6212 \\
\hline
2 & 2-methyl-heptane & 30.298 &75.489 &32.0568\\
\hline
3 & 3-methyl-heptane & 31.61&77.6835 &33.8232 \\
\hline
4 & 4-methyl-heptane & 30.194 &75.0974 &32.9665 \\
\hline
5 & 3-ethyl-hexane & 31.415 &76.9651 &34.5159 \\
\hline
6 & 2,2-dimethyl-hexane & 32.394 &99.8456 &33.9879 \\
\hline
7 & 2,3-dimethyl-hexane & 31.654 &86.9688 &35.5943 \\
\hline
8 & 2,4-dimethyl-hexane & 31.452 &86.2309 &34.114 \\
\hline
9 & 2,5-dimethyl-hexane & 31.642 &87.171 &33.3169 \\
\hline
10 & 3,3-dimethyl-hexane & 32.464 &100.2934 &36.4399 \\
\hline
11 & 3,4-dimethyl-hexane & 32.98 &89.1456 &37.1639 \\
\hline
12 & 2-methyl-3-ethyl-pentane & 31.502 &86.293 &36.0791 \\
\hline
13 & 3-methyl-3-ethyl-pentane & 33.928 &102.0731 &38.1993 \\
\hline
14 & 2,2,3-trimethyl-pentane & 33.804 &111.5061 &38.7971 \\
\hline
15 & 2,2,4-trimethyl-pentane & 31.964 &107.718 &33.9633 \\
\hline
16 & 2,3,3-trimethyl-pentane & 33.976 &111.8089 &39.7411 \\
\hline
17 & 2,3,4-trimethyl-pentane & 32.98 &98.451 &37.7739 \\
\hline
18 & 2,2,3,3-tetramethyl-butane & 34.632 &134.0448 &42.332 \\
\hline
\end{tabular}
\end{table}
\section{Correlation}
In this section we have correlation coefficients \cite{Z.Lin et.al.} of ES energy to FI and SZ energy of octane isomers. Also correlation coefficients of ES topological index and some physicochemical properties of octane isomers.
\begin{table}[h]
\centering
\caption{Correlation Coefficient of ES energy  to FI and SZ energy of octane isomers}
\begin{tabular}{|p{1cm}|p{3cm}|p{6cm}|}
\hline
Sr.No. & Energy & Correlation Coefficient \\
\hline
1& FI & 0.8900\\
\hline
2& SZ &0.9310\\
\hline
\end{tabular}
\end{table} \\
Now we are giving correlation coefficients of ES energy to some following physicochemical  properties of octain isomers.\\
viz.Boiling point(BP),Entropy, Acentric factor (AF),Enthalpy of vaporization(HVAP) and  standard enthalpy of vaporization(DHVAP)\\
\begin{table}[h]
\centering
\caption{Some Physicochemical properties of octane isomers}
\begin{tabular}{|p{1cm}|p{5cm}|p{1.5cm}|p{1.5cm}|p{1.5cm}|p{1.5cm}|p{1.5cm}|}
\hline
Sr.No. & Molecule & BP & Entropy & AF & HVAP & DHVAP \\
\hline
1 & octane & 125.7  &  111.7 & 0.3979 & 73.19 & 9.915 \\
\hline
2 & 2-methyl-heptane & 117.6 & 109.8 & 0.3792 & 70.3 & 9.484 \\
\hline
3 & 3-methyl-heptane & 118.9 & 111.3 & 0.371 & 71.3 & 9.521 \\
\hline
4 & 4-methyl-heptane & 117.7 & 109.3 & 0.3715 & 70.91 & 9.483  \\
\hline
5 & 3-ethyl-hexane & 118.5 & 109.4 & 0.3625 & 71.7 & 9.476  \\
\hline
6 & 2,2-dimethyl-hexane & 106.8&  103.4 & 0.3394 & 67.7 & 8.915   \\
\hline
7 & 2,3-dimethyl-hexane & 115.6 & 108 & 0.3483 & 70.2 & 9.272  \\
\hline
8 & 2,4-dimethyl-hexane & 109.4 & 107 & 0.3442 & 68.5 & 9.029   \\
\hline
9 & 2,5-dimethyl-hexane & 109.1 & 105.7 & 0.3568 & 68.6 & 9.051 \\
\hline
10 & 3,3-dimethyl-hexane & 112 & 104.7 & 0.3226 & 68.5 & 8.973 \\
\hline
11 & 3,4-dimethyl-hexane & 117.7 & 106.6 & 0.3404 & 70.2 & 9.316 \\
\hline
12 & 2-methyl-3-ethyl-pentane & 115.6 & 106.1 & 0.3324 & 69.7 & 9.209  \\
\hline
13 & 3-methyl-3-ethyl-pentane & 118.3 & 101.5 &0.3069 & 69.3 & 9.081 \\
\hline
14 & 2,2,3-trimethyl-pentane & 109.8 & 101.3 & 0.3001 & 67.3 & 8.826 \\
\hline
15 & 2,2,4-trimethyl-pentane & 99.24 & 104.1 & 0.3054 & 64.87 & 8.402 \\
\hline
16 & 2,3,3-trimethyl-pentane &114.8 & 102.1 & 0.2932 & 68.1 &8.897  \\
\hline
17 & 2,3,4-trimethyl-pentane & 113.5 & 102.4 & 0.3174 & 68.37 & 9.014  \\
\hline
18 & 2,2,3,3-tetramethyl-butane & 106.5 & 93.06 & 0.2553 & 66.2 & 8.41  \\
\hline
\end{tabular}
\end{table} \\ \\ \\ 
\begin{table}[h]
\centering
\caption{Correlation coefficients of BP,Entropy,AF ,HVAP and DHVAP of octane isomers with ES index}
\begin{tabular}{|p{1cm}|p{2cm}|p{2cm}|p{2cm}|p{2cm}|p{2cm}|p{2cm}|}
\hline
Sr.No. & Index & BP & Entropy & AF & HVAP & DHVAP  \\
\hline
1& ES & -0.374 & -0.8740 & -0.9003 & -0.6129 & -0.6966 \\
\hline
2& FI & -0.672 & -0.9611 & -0.9743 & -0.532 & -0.9105 \\
\hline
3& SZ & -0.248 & -0.8557 & -0.8967 & -0.5053 & -0.6139 \\
\hline
\end{tabular}
\end{table}

\newpage
\bibliographystyle{amsplain}

\end{document}